\documentclass[11pt]{article}
\usepackage{amssymb}
\usepackage{mathrsfs}
\usepackage{amsfonts}
\usepackage{tikz}

\usepackage{enumerate}

\usepackage{amsmath,amsthm,amsfonts,amssymb,amscd}

\allowdisplaybreaks[4]
\newtheorem {Problem} {Problem}[section]
\newtheorem {Theorem} [Problem]{Theorem}
\newtheorem {Lemma}[Problem]{Lemma}

\newtheorem{Definition} [Problem]{Definition}

\newtheorem {Corollary}[Problem]{Corollary}
\newenvironment {Proof}{\noindent {\bf Proof.}}{\hfill\ensuremath{\square}}
\newcommand*{\QEDB}{\hfill\ensuremath{\square}}

\newtheorem{case}{\indent Case}
\newtheorem{2case}{\indent Case}[case]
\newtheorem{3case}{\indent Case}[2case]

\begin{document}

\title{ Localized Version of Hypergraph Erd\H os-Gallai Theorem \thanks{This work is partly  supported by  National Natural Science Foundation of China (Nos. 12371354, 11971311, 12161141003),  the Science and Technology
Commission of Shanghai Municipality (No. 22JC1403600), National Key R\&D
Program of China (No. 2022YFA1006400) and the Fundamental Research
Funds for the Central Universities.
}}

\author{ Kai Zhao, Xiao-Dong Zhang\footnote{Corresponding author. E-mail: xiaodong@sjtu.edu.cn (X.-D. Zhang)}\\
\\School of Mathematical Sciences, MOE-LSC, SHL-MAC\\
Shanghai Jiao Tong University,
Shanghai 200240, P. R. China\\
Email: wkidd@sjtu.edu.cn, xiaodong@sjtu.edu.cn.\\
\\
}
\date{}
\maketitle

\begin{abstract}
This paper focuses on extensions of the classic Erd\H{o}s-Gallai Theorem  for the set of weighted function of each edge in a graph.
The weighted function of an edge $e$ of  an  $n$-vertex uniform hypergraph $\mathcal{H}$ is defined to a special function  with respect to the number of edges of the longest Berge path containing $e$.  We prove that  the summation of the weighted function of all edges is at most $n$ for an $n$-vertex uniform hypergraph $\mathcal{H}$ and characterize all extremal hypergraphs that attain the value, which  strengthens and extends the hypergraph version of  the classic Erd\H{o}s-Gallai Theorem.


\end{abstract}

\section{Introduction}

In 1959, Erd\H{o}s and Gallai proved the following classical result on the Tur\'{a}n number of paths. Let $P_k$ be a path of length $k$. Denote by  $ex(n,F)$ the maximum number of edges in an $n$-vertex simple graph $G$ which does not contain $F$ as a subgraph of $G$ for a given graph $F$.

\begin{Theorem}[Erd\H{o}s, Gallai
\cite{erdHos1959maximal}] \label{theg}
	Let $n\geq k\geq 1$ be two integers and $P_k$ be a path of length $k$. Then $ex(n,P_k)\leq \frac{n}{2}(k-1)$.
\end{Theorem}

In a hypergraph, there are many ways to define paths and cycles. Berge introduced a kind of  paths and cycles in hypergraphs, known as Berge paths and Berge cycles.

\begin{Definition}
	[Berge
\cite{berge1973graphs}] A Berge path  $P_k$ of length $k$ is an alternating sequence of distinct vertices and edges, $v_0 e_1 v_1 e_2 v_2 \dots v_{k-1} e_k v_k$ such that $v_{i-1},v_i\in e_i$ for $i=1,2,\dots,k$.   Vertices $v_0, \ldots, v_k$ are called defining vertices  and the edges $e_1, \ldots, e_k$ are called defining edges. 	Further, denote by $V(P_k)$ and $E(P_k)$  the sets of  defining vertices and defining edges of $P_k$, respectively. Moreover, denote by $P_{v_i\sim v_j}$  the segment $v_i e_{i+1} v_{i+1} e_{i+2} v_{i+1}\dots v_{j-1} e_{j} v_j$ of a Berge path $P$.

\end{Definition}

\begin{Definition}
	[Berge
 \cite{berge1973graphs}] A 
 Berge cycle  $C_k$ of length $k$ is an alternating sequence of distinct vertices and edges, $v_0 e_1 v_1 e_2 v_2 \dots v_{k-1} e_k v_0$ such that $v_{i-1},v_i\in e_i$ for $i=1,2,\dots,k$ (where indices are taken modulo $k$).  Vertices $v_0, \ldots, v_{k-1}$ are called defining vertices and the edges $e_1, \ldots, e_k$ are called defining edges, respectively. 	Moreover, denote by $V(C_k)$ and $E(C_k)$  the sets of  defining vertices and defining edges of $C_k$, respectively.

\end{Definition}
Denote $\cup_{e\in E(P)}e$  and $\cup_{e\in E(C)}e$ by $\cup E(P)$ and $\cup E(C)$, respectively.  Note that it is not necessarily $V(P)=\cup E(P)$ for a Berge path $P$ or $V(C)=\cup E(C)$ for a Berge cycle $C$. Furthermore,  denote by  $\mathcal{BP}_k$ the family of Berge paths of length $k$.  Moreover, denote by $ex_r(n, \mathcal{BP}_k)$ the maximum number of hyperedges in a simple $n$-vertex  and $r$-uniform hypergraph  which does not contain any Berge paths $P_k$ of length $k$ as a subgraph.   Gy\H{o}ri, Katona and Lemons \cite{gyHori2016hypergraph} firstly extended the Erd\H{o}s-Gallai Theorem  for Berge paths in $r$-uniform hypergraphs.

\begin{Theorem}[Gy\H{o}ri, Katona and Lemons \cite{gyHori2016hypergraph}]\label{th1}  Let $n, r, k$ be three positive integers.

(1). If  $n\geq r\geq k> 2$, then $ex_r(n,\mathcal{BP}_k)\leq \frac{n}{r+1}(k-1)$.

(2). If $n\geq k> r+1\geq 3$, then  $ex_r(n,\mathcal{BP}_k)\leq \frac{n}{k}\binom{k}{r}$.
\end{Theorem}

Later, in 2018 Davoodi, Gy\H{o}ri, Methuku and Tompkins \cite{davoodi2018erdHos} settled the remaining cases for $k=r+1$.

\begin{Theorem}[Davoodi, Gy\H{o}ri, Methuku, Tompkins \cite{davoodi2018erdHos}]
	\label{th0}
	Let  $n\geq k= r+1\geq 3$ be three positive integers. Then $ex_r(n,\mathcal{BP}_k)\leq n$.
\end{Theorem}

 Recently, in order to deeply extend the classic Tur\'{a}n  theorem, Brada\u{c} \cite{bradac2022} and  Malec and Tompkins \cite{malec2023localized} independently studied the behavior of edges in an arbitrary $n$-vertex graph $G$  for a given weighted function on edge set. Malec and Tompkins  
  \cite{malec2023localized} defined that  $w(e)$ is  the maximum value $r$ such that $e$ occurs in a subgraph of $G$ isomorphic to a complete graph $K_r$.
  They proved the following results.

  \begin{Theorem}[Brada\u{c} \cite{bradac2022}, and  Malec and Tompkins \cite{malec2023localized}]\label{turan} Let $G$ be an $n$-vertex graph. Then
  $$\sum_{e\in E(G)}\frac{w(e)}{w(e)-1}\le \frac{n^2}{2}.$$
  \end{Theorem}

 Theorem \ref{turan} is referred to as a localized version of the Tur\'{a}n  theorem  since the weighted function only depends on the structure of each edge.  The technique of weighted function has also been used to study the Ramsey-Tur\'{a}n type  problems (see \cite{balogh2022} and \cite{balogh2023}).   In 2023, Malec and Tompkins \cite{malec2023localized} obtained a local version of the Erd\H{o}s-Gallai theorem  and proved the following  result.
\begin{Theorem}[Malec and Tompkins \cite{malec2023localized}]
	\label{thdm}
	 Let $G$ be an $n$-vertex simple  graph. Then
	\begin{equation}
		\sum_{e\in E(G)}\frac{2}{p(e)}\leq n,
\end{equation}
where $p(e)$ is the maximum $k$ such that $e$ occurs in a subgraph of $G$ isomorphic to $P_k$.
		Moreover, the equality holds if and only if each component is a clique.
\end{Theorem}

 Motivated by the  Erd\H{o}s-Gallai theorem for hypergraphs and the localized version  for simple graphs,  we consider a localized version of   the  Erd\H{o}s-Gallai theorem for hypergraphs.

This paper is organized as follows. In Section 2, the main results and corollaries of the paper have been presented after some notations and symbols are introduced. In Section 3, two key theorems on the structure of hypergraphs involving a weighted function have been proved, which are used to give a proof of the main results in Section 4.

\section{Notations and main results}

Let $\mathcal{H}=(\mathcal{V},\mathcal{E})$  be a hypergraph with vertex set $\mathcal{V}$  and edge set $\mathcal{E}$  which is  the family of non-empty subsets of $\mathcal{V}$.  If  each edge $e\in \mathcal{E}$  is a  $r$-element subset of $\mathcal{V}$, then  $\mathcal{H}$  is called an $r$-uniform hypergraph, or simply $r$-graph. A hypergraph $\mathcal{H}$ is said to be connected if for each pair of distinct vertices $u,v\in V(\mathcal{H})$, there exists a Berge path whose terminal vertices  are $u$ and $v$.   Moreover, if
$S\subseteq V(\mathcal{H})$ and $\mathcal{F}\subseteq E(\mathcal{H})$, denote  $N_{\mathcal{F}}(S)$  by  $\{e\in\mathcal{F}:e\cap S\neq\emptyset\}$, and the subgraph $\mathcal{H}\setminus S$ by $ (V(\mathcal{H})\setminus S,\{e\in E(\mathcal{H}): e\cap S=\emptyset\})$, respectively.
If there are no ambiguous, we use $N(S)$ instead of $N_{E(\mathcal{H})}(S)$, and $N_{\mathcal{F}}(v)$ instead of $N_{\mathcal{F}}(\{v\})$ for convenience.
It is easy to see that  $N_{\mathcal{F}}(A\cup B)=N_{\mathcal{F}}(A)\cup N_{\mathcal{F}}(B)$ for $A,B\subseteq V(\mathcal{H})$, and  $\{E(\mathcal{H}\setminus S),N(S)\}$ is an edge partition of $E(\mathcal{H})$.

For an $r$-uniform hypergraph $\mathcal{H}$, define a weight on  each edge $e\in E(\mathcal{H})$,
\begin{equation*}
	p_{\mathcal{H}}(e):=\max\{k: e\in E(P) \text{ for some  Berge path} \ P \ \text{ of length}\ k  \ \text{ in }\mathcal{H}\}.
\end{equation*}
Moreover, for a given $r\ge 3$, we define  a function with respect to  real $x\ge 1$ as follows.
\begin{equation}
	f_r(x)=\left\{
	\begin{array}{lll}
		&\frac{1}{r}, & {\rm if }\  x=1,\\
		&\frac{x}{r+1}, & {\rm if} \ 1<x\leq r-1,\\
		&\frac{1}{r}\binom{x}{r-1}, & {\rm if} \ x\geq r.
	\end{array}
	\right.
\end{equation}
It is easy to see that the function $f_r(x)$ is strictly increasing with respect to $x$  if  $r\ge 3$.  For each edge $e$ of an $r$-uniform hypergraph $\mathcal{H}$,  the weighted function on $e$ is defined to be $f_r(p_{\mathcal{H}}(e))$. In this paper,   the main result of this paper can be stated as follows.

\begin{Theorem}
	\label{th3conn}
	Let $\mathcal{H}$ be an  $n$-vertex, $r$-uniformly connected hypergraph with $n\ge r\ge 3$. Then
	\begin{equation}
		\label{eq0}
		\sum_{e\in E(\mathcal{H})}\frac{1}{f_r(p_{\mathcal{H}}(e))}\leq n.
	\end{equation}
	Moreover,  equality holds if and only if one of the following holds:
	\begin{enumerate}[(i)]
		\item $n=r+1$ and $|E(\mathcal{H})|\in \{1, 2, \ldots, r-1, r+1\}$
		\item $n\neq r+1$ and $\mathcal{H}$ is an $n$-vertex, $r$-uniform complete hypergraph, i.e., $|E(\mathcal{H})|=\binom{n}{r}$ for $n\neq r+1$.
	\end{enumerate}
\end{Theorem}
The proof of Theorem \ref{th3conn} will be postponed in Section 4.  By Theorem \ref{th3conn},  we can deduce the following results   and present a new proof of Theorems \ref{th1} and \ref{th0}.

\begin{Corollary}
	\label{th3}
	Let $\mathcal{H}$ be an $n$-vertex, $r$-uniform hypergraph. Then
	\begin{equation}
		\label{eq00}
		\sum_{e\in E(\mathcal{H})}\frac{1}{f_r(p_{\mathcal{H}}(e))}\leq n.
	\end{equation}
	Moreover, equality holds if and only if each component $\mathcal{C}$ of $\mathcal{H}$ satisfies one of the following:
	\begin{enumerate}[(i)]
		\item $|V(\mathcal{C})|=r+1$ and $|E(\mathcal{C})|\in \{1, 2,\ldots, r-1, r+1\}$.
		\item $|V(\mathcal{C})|\neq r+1$ and $\mathcal{C}$ is an $r$-uniform complete hypergraph. 
	\end{enumerate}
\end{Corollary}

\begin{Proof}
 Let $\mathcal{C}_1,\mathcal{C}_2,\dots,\mathcal{C}_t$ be the components of $\mathcal{H}$.
By Theorem \ref{th3conn},  for $i=1, 2, \ldots, t$,
\begin{equation}
	\label{cpnt}
	\sum_{e\in E(\mathcal{C}_i)}\frac{1}{f_r(p_{\mathcal{H}}(e))}\leq |V(\mathcal{C}_i)|. 
\end{equation}
Hence,
\begin{equation}
	\label{cp}
	\sum_{e\in E(\mathcal{H})}\frac{1}{f_r(p_{\mathcal{H}}(e))}=\sum_{i=1}^{t}\sum_{e\in E(\mathcal{C}_i)}\frac{1}{f_r(p_{\mathcal{H}}(e))}\leq \sum_{i=1}^{t}|V(\mathcal{C}_i)|=n.
\end{equation}
So \eqref{eq00} holds. Moreover,  equality \eqref{eq00} holds if and only if  equality  \eqref{cpnt}  holds.  Therefore the assertion holds.
\end{Proof}

\begin{Corollary}\cite{ davoodi2018erdHos,gyHori2016hypergraph}\label{cor3}
Let $n, r, k$ be three positive integers.

(1). If  $n\geq r\geq k> 2$, then $ex_r(n,\mathcal{BP}_k)\leq \frac{n}{r+1}(k-1)$.

(2). If $n\geq k\ge r+1\geq 3$, then  $ex_r(n,\mathcal{BP}_k)\leq \frac{n}{k}\binom{k}{r}$.

\end{Corollary}

\begin{Proof} If $r=2,$ the assertion follows from Erd\H{o}s-Gallai Theorem. So we assume that $r\ge 3$ and let $\mathcal{H}$ be  an $n$-vertex, $r$-uniform  hypergraph which does not contain any Berge path of length $k$.  Then  $p_{\mathcal{H}}(e)\leq {k-1}$ for each $e\in E(\mathcal{H})$.  Since $f_r(x)$  is strictly increasing with respect to $x$,  $f_r(p_{\mathcal{H}}(e))\leq f_r(k-1)$ for each $e\in E(\mathcal{H})$. By Corollary \ref{th3}, we have
\begin{equation}
	\frac{|E(\mathcal{H})|}{f_r(k-1)}=\sum_{e\in E(\mathcal{H})}\frac{1}{f_r(k-1)}\leq\sum_{e\in E(\mathcal{H})}\frac{1}{f_r(p_{\mathcal{H}}(e))}\leq n. \nonumber
\end{equation}
Hence $|E(\mathcal{H})|\leq n f_r(k-1)$ and  
$$ex_r(n,\mathcal{BP}_k)\leq nf_r(k-1)=\left\{ \begin{array}{ll}
\frac{n(k-1)}{r+1}, & 2<k\le r\le n,\\
\frac{n}{k}\binom{k}{r}, & n\ge k\ge r+1.\end{array}\right.
$$
So  the assertion holds.
\end{Proof}

\section{ Two key technique Theorems }

In order to prove  Theorem \ref{th3conn}, we  only consider $r\ge 3$  in later and need the following definition.
\begin{Definition}
	Let $\mathcal{H}$ be an $n$-vertex, $r$-uniform hypergraph with a longest Berge path of length $k$.
A  nonempty subset $S$ of $V(\mathcal{H})$ is said to be  {\it good} if $p_{\mathcal{H}}(e)=k$ for each $e\in N(S)$ and $|N(S)|\le f_r(k)|S|$. Moreover
denote by $\mathcal{S}_r(\mathcal{H})$ the family of all good sets of $V(\mathcal{H})$, i.e.,
	\begin{equation}
		\mathcal{S}_r(\mathcal{H}):=\{S\subseteq V(\mathcal{H}): S\neq\emptyset; \forall e\in N(S),p_{\mathcal{H}}(e)=k; |N(S)|\leq f_r(k)|S|\}.
	\end{equation}
\end{Definition}
For example,  there is at least a good set  $V(\mathcal{K}_n^r)$ in  an $n$-vertex, $r$-uniform complete hypergraph $\mathcal{K}_n^r$. Because it is easy to see that
$p_{\mathcal{H}}(e)=n-1$ for each $e\in N(V(\mathcal{K}_n^r))$ and  $|N(V(\mathcal{K}_n^r))|=|E(\mathcal{K}_n^r)|=\binom{n}{r}= f_r(n-1)|V(\mathcal{K}_n^r)|$, where $k=n-1$.
In fact,  we will prove in this section that there is at least a good set $S$ in an $n$-vertex, $r$-uniform hypergraph $\mathcal{H}$  with the longest Berge path of length $k$ for $r\ge 3$, i.e., $\mathcal{S}_r(\mathcal{H})\neq \emptyset$.

\begin{Lemma}
	\label{le4}
	Let $\mathcal{H}$ be a connected  hypergraph with a longest Berge path of length $k$.  If there exists a Berge cycle $C$ of length  $k+1$, then $V(\mathcal{H})= V (C)$. Moreover,  $p_{\mathcal{H}}(e)=k$ for each $ e\in E(\mathcal{H})$.
\end{Lemma}

\begin{proof} Let $C=v_0 e_1 v_1 e_2 v_2 \dots v_k e_{k+1} v_0$ be a Berge cycle of length  $k+1$. Then   $e_1 \cup e_2 \cup \dots \cup e_{k+1} \subseteq V(C)=\{v_0, v_1,\ldots, v_k\}$. In fact, if there exists a vertex $u\in e_i$ and $u\notin V(C)$ for $1\le i\le k$, then $ue_iv_{i+1}e_{i+1}\cdots e_{k+1}v_0e_1v_1\cdots e_{i-1}v_{i-1}$ is a Berge path of length $k+1$, which is a contradiction.  Furthermore, if there exists a vertex $w\in  V(\mathcal{H})$ and $w\notin V(C)$, then there exists a Berge path $Q=wh_1u_2\cdots u_{l-1}h_lv_i$ in $\mathcal{H}$ such that $u_2, \ldots, u_{l-1}\notin V(C)$ and $v_i\in V(C)$, since $\mathcal{H}$ is connected. Then   there is  a Berge path $W=wh_1u_2\cdots u_{l-1}h_lv_ie_{i+1}v_{i+1}\cdots v_ke_{k+1}v_0e_1\cdots v_{i-1}$ of length $k+1$  in $\mathcal{H}$, which is a contradiction. Hence
	 $V(\mathcal{H})=V(C)$.

In addition,  since $v_{i-1} e_i v_i e_{i+1} v_{i+1} \dots v_k e_{k+1} v_0 e_1 \dots e_{i-2} v_{i-2}$ is a Berge path of $\mathcal{H}$ containing $e_i$,  $p_{\mathcal{H}}(e_i)=k$ for $i=1, \ldots,k+1.$  For $e \in E(\mathcal{H})\setminus E(C)$, there exist two vertices $v_i, v_j\in e$ since  $V(\mathcal{H})=V(C)$. Hence there is a Berge path $v_{j+1} e_{j+2} \dots e_{k+1} v_0 e_1 v_1 \dots v_{i-1} e_i v_i e v_j e_j v_{j-1} \dots v_{i+2} e_{i+2} v_{i+1} $ of length $k$. Hence $p_{\mathcal{H}}(e_i)$ $=k$. So the assertion holds.
	\end{proof}

\begin{Corollary}
	\label{le5}
	Let $\mathcal{H}$ be an $n$-vertex, $r$-uniformly connected  hypergraph with the longest Berge path of length $r$. If $\mathcal{H}$ contains a Berge cycle $C$ of length $r+1$, then $\mathcal{H}$ must be an $(r+1)$-vertex complete hypergraph.
\end{Corollary}
\begin{proof} Let   $C=v_0 e_1 v_1 e_2 v_2 \dots v_r e_{r+1} v_0$ be a Berge cycle of length  $r+1$. Then by Lemma \ref{le4},  $V(\mathcal{H})= V (C)=\{v_0, \ldots, v_r\}$. So $|V(\mathcal{H})|=r+1$. In addition, since $|E(\mathcal{H})|\leq\binom{|V(\mathcal{H})|}{r}=r+1$ and $|E(\mathcal{H})|\geq |E(C)|=r+1$, we obtain $|E(\mathcal{H})|=r+1=|\binom{V(\mathcal{H})}{r}|$. Hence $\mathcal{H}$ is  a complete hypergraph.
\end{proof}

\begin{Corollary}
	\label{l5c1}
	Let $\mathcal{H}$ be an $r$-uniformly connected  hypergraph with a longest Berge path of length  $r$. If there exists a Berge path  $P=v_0 e_1 v_1 e_2 v_2 \dots v_{r-1} e_r v_r$ such that $N(v_0)\setminus E(P) \neq \emptyset$, then $\mathcal{H}$ must be an $(r+1)$-vertex complete hypergraph.
\end{Corollary}
\begin{proof}
Since	$N(v_0)\setminus E(P) \neq \emptyset$, there exists a hyperedge $h\in N(v_0)\setminus E(P)$. If $v_r\in h$, then there exists a Berge cycle of length $r+1$ and the assertion holds by Corollary \ref{le5}. We now assume that  $v_r\notin h$. Then $h\subseteq\{v_0,v_1,\dots,v_{r-1}\}$, otherwise $\mathcal{H}$  has a Berge path of length at least $r+1$. Hence, $h=\{v_0,v_1,\dots,v_{r-1}\}$ by $|h|=r$.  It is easy to see that there is another Berge path $P^{\prime}=v_{r-2} e_{r-2} v_{r-3} \dots v_1 e_1 v_0 h v_{r-1} e_r v_r$  of length $r$ with $ V (P')= V (P)$, and $e_{r-1}\in N(v_{r-2})\setminus E(P^{\prime})$.  Furthermore,  $v_r\in e_{r-1}$,  otherwise $e_{r-1}=\{v_0,v_1,\ldots,v_{r-1}\}=h$ which contradicts to $h\notin E(P)$.   Therefore, there is a Berge cycle $C=v_r e_{r-1} v_{r-2} v_{r-2} e_{r-2} v_{r-3} \dots v_1 e_1 v_0 h v_{r-1} e_r v_r$ of length $r+1$. So the assertion holds by Corollary \ref{le5}.
\end{proof}

\begin{Corollary}
	\label{l5c2}
	Let $\mathcal{H}$ be an  $r$-uniformly connected hypergraph with a longest Berge path of length  $r$. If  there is  a Berge cycle of length at least $r$ in $\mathcal{H}$, then $\mathcal{H}$ must be an $(r+1)$-vertex complete hypergraph.
\end{Corollary}
\begin{proof} If  $\mathcal{H}$ contains a Berge cycle of length at least $r+1$, then by Corollary \ref{le5}, the assertion holds. Now we assume that there is a Berge cycle  $C=v_0 e_1 v_1 \ldots v _{r-1} e_r v_0 $  of length $r$ in $\mathcal{H}$.  Since there is a Berge path of length $r$ in a connected hypergraph $\mathcal{H}$, there exists a vertex  $ u\notin V(C)$  and a hyperedge $h$ in $\mathcal{H}$ such that $\{u,v_i\}\subseteq h$, $0\le i\le r-1$.
 Then  $P=v_{i+1}e_{i+1}\ldots v_{r-1}e_rv_0\ldots e_iv_ihu$ is a Berge path of length $r$ with $e_i\in N(v_{i+1})$ and $e_{i+1}\notin E(P)$.
By Corollary \ref{l5c1}, the assertion holds.
\end{proof}

\begin{Lemma}
	\label{le3}
	Let $\mathcal{H}$ be an $r$-uniform hypergraph satisfying  the  property that for each longest path,  each edge containing  a terminal vertex of the path belongs to its defining edge set. Let  $ P=v_0 e_1 v_1 e_2 v_2 \ldots v_{k-1} e_k v_k$ be a longest Berge path with $W= V (P)\setminus\{v_0\}=\{v_1,v_2,\dots,v_k\}$. Then the following hold.
	
	\begin{enumerate}[(i)]
		\item \label{l3c1} If  $v_i\in e_1\cap W$ with  $e_i\setminus W=\emptyset$ for some $1\le i\le k$, then there exists a Berge cycle of length at least $r$.
		\item \label{l3c1.5} If $v_i\in e_1\cap W$ with $e_i\setminus W\neq\emptyset$ for some  $1\le i\le k$, then $e_i$ is the first edge of some longest Berge path, 
and $N(e_i\setminus W)\subseteq  E(P) $.
		\item \label{l3c2} If $v_i\in e_1\cap W$ with $(e_i \cap e_j) \setminus W \neq \emptyset$   for some $1\le i<j\le k$, then $N(v_{j-1})\subseteq  E(P) $.
		\item \label{l3c2.5} If   $v_j\in e_1\cap W$ with $(e_i \cap e_j) \setminus W \neq \emptyset$   for $1\le i<j\le k$, then $N(v_i)\subseteq  E(P) $.
	\end{enumerate}
\end{Lemma}

\begin{proof}
	
	(\ref{l3c1}) Since $v_i\in e_1\cap W$ with  $e_i\setminus W=\emptyset$ for  $1\le i\le k$, we have  $e_i\subseteq W$ and $|e_i\cap W|=|e_i|=r$. Then
$t:=\max\{j: v_j\in e_i\cap W\}\ge r$. Hence there is a Berge cycle  $C= v_t e_1 v_1 \dots v_{i-1} e_i v_t e_t v_{t-1} \dots v_{i+1} e_{i+1} v_i$ of length  $t\ge r$.  So (i) holds. 
	
	(\ref{l3c1.5}). 
	Since $e_i\setminus W\neq \emptyset,$  let $u$ be any vertex in $e_i\setminus W$. Then  by $v_i\in e_1$, there is a  Berge path $$Q=  ue_i v_{i-1} e_{i-1} v_{i-2} \ldots v_1 e_1 v_i e_{i+1} v_{i+1} \ldots v_{k-1} e_k v_k$$ of length $k$ with a terminal vertex $u$.  By the property of $\mathcal{H}$,   each edge containing a terminal vertex in a longest path  must be contained in the defining edge set. Hence $N(u)\subseteq  E(P) $  and  $N(e_i\setminus W)=\bigcup_{u\in e_i\setminus W}N(u)\subseteq  E(P) $. So (ii) holds.
	
	(\ref{l3c2}). Since $(e_i \cap e_j) \setminus W \neq \emptyset$, there is a vertex 
  $u\in (e_i \cap e_j)\setminus W$. Thus  $u\in e_i, u\in e_j$ and $v_i\in e_1$.
	Hence there is a Berge path $$Q=v_{j-1} e_{j-1} v_{j-2} \ldots v_{i+1} e_{i+1} v_i e_1 v_1 e_2 v_2 \ldots v_{i-1} e_i u e_{j} v_j e_{j+1} v_{j+1} \ldots v_{k-1} e_k v_k$$ of length $k$ with a terminal vertex  $v_{j-1}$. Hence  $N(v_{j-1})\subseteq  E(P) $. So (iii) holds.
	
	(\ref{l3c2.5}).
 Since $(e_i \cap e_j) \setminus W \neq \emptyset$, there is a vertex 
  $u\in (e_i \cap e_j)\setminus W$.
	Hence there is a Berge path   $$Q=v_i e_{i+1} v_{i+1} \ldots v_{j-1} e_{j} u e_i v_{i-1} \ldots v_1 e_1 v_j e_{j+1} v_{j+1} \ldots v_{k-1} e_k v_k$$ of length $k$ with a terminal vertex $v_{i}$. Hence, $N(v_{i})\subseteq  E(P) $. So (iv) holds.
\end{proof}

\begin{Theorem}
	\label{thx}
	Let $\mathcal{H}$  be  an $n$-vertex, $r$-uniformly connected  hypergraph  with the  longest Berge path of length $k$.  If  $1<k\leq r$, then  there exists a nonempty set $S\subseteq V(\mathcal{H})$ such that  $p_{\mathcal{H}}(e)=k$ for  each $ e\in N(S)$, where $S$ satisfies  one of the following:
	\begin{enumerate}[(i)]
		\item \label{x0} $k=r$, $|S|= r$ and $|N(S)|\le r$;
		\item \label{x1} $k=r$, $|S|\geq r+1$ and $|N(S)|\leq r+1$;
		\item \label{x2} $1<k<r$, $|S|\geq r+1$ and $|N(S)|\leq k$;
		\item \label{x3} $|S|=r-1$, and $|N(S)|=1$.
	\end{enumerate}
	Furthermore, there is a good set $S$ in $\mathcal{H}$,  and  $\mathcal{S}_r(\mathcal{H})\neq \emptyset$.
\end{Theorem}

\begin{proof} Let $\mathcal{BP}_k$ be the set of all Berge paths of length $k$ in $\mathcal{H}$. We consider the following two cases.

\setcounter{case}{0}
\begin{case} There exists a Berge path $P=v_0 e_1 v_1 e_2 v_2 \ldots v_{k-1} e_k v_k \in \mathcal{BP}_k$  such that  $N(v_0)\setminus E(P)\neq \emptyset$.
\end{case}

Then there exists a hyperedge $e\in N(v_0)$ and $e\notin E(P)$.  Since   $P$ is a longest Berge path of length $k$, we have $e\subseteq V(P)$ by $e\in N(v_0)$.  Then $r=|e|\le |V(P)|=k+1\le r+1$, which implies that either $k=r$ or $k=r-1$.  Furthermore, we claim that  $k=r$. In fact, if $k= r-1$, then $r=k+1>2$ and $e=V(P)$.  Since $|e_k|=r$ and $e_k\neq e$,  there exists a vertex $u\in e_k$ and $u\notin V(P)$. So there exists a Berge path $Q=ue_kv_kev_0e_1v_1\ldots v_{k-2}e_{k-1}v_{k-1}$ of length $k+1$, which is impossible. Hence $k=r$.  Furthermore, by Corollary \ref{l5c1}, $\mathcal{H}$ must be an $(r+1)$-vertex complete hypergraph.  Set $S=V(\mathcal{H})$. It is easy to see that  $p_{\mathcal{H}}(e)=k$ for $\forall e\in N(S)$, where $|S|=r+1$ and $|N(S)|=|E(\mathcal{H})|=r+1$.  Hence, $S$ satisfies  (\ref{x1}) in Theorem \ref{thx}. Moreover, $|N(S)|=f_r(r)|S|$.
So $S$ is a good set and  $\mathcal{S}_r(\mathcal{H})\neq \emptyset$.
\begin{case}
	\label{cs2}
	For each $P=v_0 e_1 v_1 e_2 v_2 \cdots v_{k-1} e_k v_k \in \mathcal{BP}_k$,  we have $N(v_0)\subseteq E(P)$ and $N(v_k)\subseteq E(P)$.
\end{case}
 Without loss of generality, we may assume that $P=v_0 e_1 v_1 e_2 v_2 \cdots v_{k-1} e_k v_k \in \mathcal{BP}_k$  has the minimum size  $|e_1\setminus W|$ in $\mathcal{BP}_k$,  where $W= V (P)\setminus\{v_0\}=\{v_1,v_2,\dots,v_k\}$.
       Clearly $v_0\in e_1 $ and $v_0\notin  W$. Then $e_1\setminus W\neq \emptyset$ and $v_1\in e_1\cap W$. Hence by   Lemma \ref{le3} (\ref{l3c1.5}),  $N(e_1\setminus W)\subseteq  E(P) $. Now we consider the following two subcases.

\begin{2case}
	$|e_1\cap W|=1$.
\end{2case}

If there is 		$(e_1 \cap e_j)\setminus W \neq\emptyset$ for some $2\le j\le k$, then  $N(v_{i-1})\subseteq  E(P) $ by 	 Lemma \ref{le3} (\ref{l3c2}).  Let $S:=\{v_{j-1},v_k\}\cup (e_1\setminus W)$. Clearly, $v_{j-1}, v_k\in W$.  Moreover,  $|e_1\setminus W|=|e_1|-|e_1\cap W|=r-1$. Hence  $|S|= |e_1\setminus W|+2=r+1$. Furthermore, $N(S)=N(v_{j-1})\cup N(v_k)\cup N(e_1\setminus W)\subseteq E(P)$. Then  $p_{\mathcal{H}}(e)=k$  for $e\in N(S)$,  $|S|=r+1$ and  $|N(S)|\leq | E(P) |=k$. Hence $S$ satisfies  either  (\ref{x1}) or (\ref{x2}) in Theorem \ref{thx}. Furthermore, it is easy to see that $|N(S)|\le k\le f_r(k)|S|$. So $S$ is a good set and $\mathcal{S}_r(\mathcal{H})\neq \emptyset$.
	
If 		$(e_1 \cap e_j)\setminus W =\emptyset$ for  $j=2, \ldots, k$, then $e_j\cap (e_1\setminus W)=\emptyset$ and  $e_j\notin N(e_1\setminus W)$ for $j=2, \ldots, k.$   On the other hand,  $N(e_1\setminus W)\subseteq  E(P)=\{e_1, \ldots, e_k\}$  and $e_1\in N(e_1\setminus W)$. Then  $N(e_1\setminus W)=\{e_1\}$.  Set  $S=e_1\setminus W$. Then $|S|=r-1$ and  $|N(S)|=1$. So  $S$ satisfies  (\ref{x3}) in Theorem \ref{thx}.
Furthermore, by $1<k\le r$, we have $|N(S)|=1\le f_r(k)|S|$ and $S$ is a good set.
\begin{2case}
	$|e_1\cap W|>1$. 
\end{2case}

  Denote  $e_1\cap W=\{v_{i_0},v_{i_1},\dots,v_{i_s}\}$ with $s\ge 1$ and $1=i_0<i_1<\dots<i_s\le k$.  If  there is 	$e_{i_j}\setminus W=\emptyset$ for some $1\le j\le s$,
then by Lemma \ref{le3} (\ref{l3c1}), there exists a Berge cycle of length at least $r$. Hence by Corollary \ref{l5c2},  $\mathcal{H}$ must be an $(r+1)$-vertex complete hypergraph. So $k=r$.  Let  $S=V(\mathcal{H})$. Then  it is easy to see that   $p_{\mathcal{H}}(e)=k$ for  each $ e\in N(S)$. Moreover  $|S|=r+1$ and $|N(S)|= E(\mathcal{H})|=r+1$. Hence $S$ satisfies   (\ref{x1}) in  Theorem \ref{thx}.
Furthermore, $|N(S)|=r+1\le f_r(k)|S|$ and $S$ is a good set.

  Therefore, we assume now that 	$e_{i_j}\setminus W \neq\emptyset$ for  $j=1, \ldots, s$.
In addition,  by $v_0\in e_1\setminus W$,   we have $|e_1\setminus W|\neq\emptyset$.  Hence by Lemma \ref{le3} (\ref{l3c1.5}),  $e_{i_j}$ is the first edge of some Berge path in $\mathcal{BP}_k$ for $j=0, 1, \ldots, s$.  Then by the minimality of $|e_1\setminus W|$, we have
\begin{equation}
	\label{eq4}
	|e_{i_j}\setminus W|\geq |e_{i_0}\setminus W|= |e_1\setminus W|\geq 1, \ \ j=1, 2,\ldots, s.
\end{equation}
We define the sets $A_0,A_1,A_2,\dots,A_s$  recursively.  Set  $A_0=e_{i_0}\setminus W$.  For $j=1,2,\dots,s$, set
\begin{equation}\label{eq10}
	A_j=\left\{
	\begin{array}{lll}
		&A_{j-1}\cup (e_{i_j}\setminus W), &  {\text{if}}\ A_{j-1}\cap (e_{i_j}\setminus W)=\emptyset; \\
		&A_{j-1}\cup (e_{i_j}\setminus W)\cup \{v_{i_j-1}\}, & {\text{otherwise}}.
	\end{array}
	\right.
\end{equation}
Then we have the following claims:

(a). $A_j\cap W\subseteq \{v_{i_1-1}, v_{i_2-1}, \ldots, v_{i_j-1}\}$  and $A_j\setminus W\subseteq (e_{i_0}\setminus W)\cup \cdots \cup (e_{i_j}\setminus W)$ for $j=1, \ldots, s$.

(b). $N(A_j)\subseteq E(P)$ for $j=1, \ldots, s$.

(c). $|A_j|\ge |A_{j-1}|+1$ for $j=1, \ldots, s$.

(d). $|A_s|\ge r-1$ with equality if and only if $|A_j|=|A_{j-1}|+1$ for $j=1, \ldots, s$.

 We prove (a)-(c) by the induction on $j$. Clearly, (a)-(c) holds for $j=1$ and we assume that the assertion holds for $j-1$.

 If $A_{j-1}\cap (e_{i_j}\setminus W)=\emptyset$, then $A_j\cap W=(A_{j-1}\cup (e_{i_j}\setminus W))\cap W=A_{j-1}\cap W
\subseteq \{v_{i_1-1}, v_{i_2-1}, \ldots, v_{i_{j-1}-1}\}$. Moreover, by
$v_{i_j}\in e_1\cap W$, $e_{i_j}\setminus W\neq \emptyset$ and
Lemma \ref{le3} (\ref{l3c1.5}), we have $N(e_{i_j}\setminus W)\subseteq E(P)$.
 So $N(A_j)\subseteq N(A_{j-1})\cup N(e_{i_j}\setminus W)\subseteq E(P)$.  In addition,  by $e_{i_j}\setminus W\neq\emptyset$, we have
$|A_j|=|A_{j-1}|+|e_{i_j}\setminus W|\ge |A_{j-1}|+1$. Hence, (a)-(c) holds.

 If $A_{j-1}\cap (e_{i_j}\setminus W)\neq\emptyset$, then  $A_j\cap W=((A_{j-1}\cup (e_{i_j}\setminus W))\cup \{v_{i_j-1}\})\cap W=A_{j-1} \cup \{v_{i_j-1}\}\subseteq \{v_{i_1-1}, v_{i_2-1}, \ldots, v_{i_{j}-1}\}$. Moreover, by  Lemma \ref{le3} (\ref{l3c2}), we have   $N(v_{i_j-1})\subseteq E(P)$.
  Hence,  $N(A_j)\subseteq N(A_{j-1})\cup N(e_{i_j}\setminus W)\cup N(v_{i_j-1})\subseteq E(P) $.  In addition, by $v_{i_j-1}\in W$, we have $v_{i_j-1}\notin A_{j-1}$. Hence,
 $|A_j|\ge |A_{j-1}|+1$.  So  (a)-(c) hold. 

In addition, by (c) and $r=|e_1\setminus W|+|e_1\cap W|=s+1+|A_0|$, we have $|A_s|\ge |A_{s-1}|+1\ge\cdots\ge |A_0|+s=r-1$ with equality if and only if $|A_j|=|A_{j-1}|+1$ for $j=1, \ldots, s$. So (d) holds. We finish the proof of Claims (a)-(d).

Furthermore,  if $|A_s|\ge r$,  let $S=A_s\cup \{v_k\}. $  It is easy to see that  $|S|\ge r+1$ and   $p_{\mathcal{H}}(e)=k$ for  each $ e\in N(S)$.  Moreover by the assumption of Case 2, $N(S)\subseteq N(A_s \cup \{u_k\})\subseteq N(A_s)\cup N(v_k)\subseteq E(P)$, which implies that $|N(S)|\le |E(P)|=k$. Hence $S$ satisfies  (\ref{x1}) or (\ref{x2}). 
So $|N(S)|\le f_r(k)|S|$ and $S$ is a good set.
Hence, we can assume that $|A_s|=r-1$, i.e., $|A_j|=|A_{j-1}|+1$ for $j=1, \ldots, s$ and consider the following two subcases.

\begin{3case}
	$A_{j-1}\cap (e_{i_j}\setminus W)=\emptyset$ for $j=1, \ldots, s$.
\end{3case}

Then by (a), $A_j=(e_{i_0}\setminus W)\cup \cdots \cup (e_{i_j}\setminus W)$ and $((e_{i_0}\setminus W)\cup \cdots\cap (e_{i_{j-1}}\setminus W))\cap (e_{i_j}\setminus W)=\emptyset$. Hence, by (d), $|e_{i_j}\setminus W|=1$ for $j=1, \ldots, s$. In addition, by (9), $|e_{i_0}\setminus W|=1$, which implies that $s+1=|e_{i_0}\setminus W|=|e_{i_0}|-|e_{i_0}\cap W|=r-1$.
 Hence, $$r\ge k=|W|\ge |e_{i_j}\cap W|=|e_{i_j}|-|e_{i_j}\setminus W|=r-1, \ {\text{ for}}\ j=0, \ldots, r-2.$$

 If $k=r$, then let $S=A_s\cup \{v_k\}$.  Then $|S|=r$ and $N(S)\subseteq N(A_s)\cup N(v_k)\subseteq E(P)$ which implies that (\ref{x0})  holds.
 So $|N(S)|\le f_r(r)|S|$ and $S$ is a good set.

If $k=r-1$, then $|W|=k=r-1$ and $|e_{i_j}\cap W|=r-1$ for $j=0, \ldots, s.$ So $e_{i_j}\cap W=W$ for $j=0,
\ldots, s$. Then $\{v_{i_0}, \ldots, v_{i_s}\}=e_1\cap W=W=\{v_1, \ldots, v_{r-1}\}$  and  $i_j=j+1$ for $j=0, \ldots, s=r-2$.  So  let $e_{i}\setminus W=\{u_{i}\}$ for $i=1, \ldots, r-1$. Then  there is a Berge path of length $r-1$ with one terminal vertex $v_i$:   $u_{i+1} e_{i+1} \dots e_{r-1} v_{r-1} e_1 v_1$ $ e_2 v_2 \dots v_{i-1} e_i v_i$ for $i=1, \ldots, r-1$. Hence, $N(W)\subseteq  E(P) $. Let $S=W\cup A_s$. Then $|S|\geq 2(r-1)\geq r+1$ and $|N(S)|\le |E(P)|\le k=r-1$ by $r\ge 3$. Thus, (\ref{x2})  holds. So $|N(S)|\le k\le \frac{k}{r+1}|S|=f_r(k)|S|$ and $S$ is a good set.

\begin{3case}
	There exists an $1\le p\le s$ such that $A_{p-1}\cap (e_{i_p}\setminus W)\neq\emptyset$ and    $A_{j-1}\cap (e_{i_j}\setminus W)=\emptyset$ for $j=1, \ldots, p-1$.
\end{3case}

Then by (\ref{eq10}), $A_{p-1}=(e_{i_0}\setminus W)\cup\cdots\cup  (e_{i_{p-1}}\setminus W)$ and  $(e_{i_q}\setminus W)\cap  (e_{i_{p}}\setminus W)\neq \emptyset$ for some $0\le q<p$.  If  $i_q<i_p-1$, then
 $v_{i_q}\notin A_s\cup \{v_k\}$, which implies that   (\ref{x1}) or (\ref{x2}) 
 holds for $S=A_s\cup \{v_{i_q}, v_k\}$. Moreover, $S$ is a good set.
If 	$i_q=i_p-1$ and $i_q>1$,  let  $x\in (e_{i_q}\setminus W)\cap (e_{i_p}\setminus W)$. Hence, the Berge path $v_{i_q-1} e_{i_q-1} v_{i_q-2} \dots v_1 e_1 v_{i_q} e_{i_q} x e_{i_p} v_{i_p} \dots v_{k-1} e_k v_k$ is a maximum length of length $k$. So $N(v_{i_q-1})\subseteq  E(P) $ by assumption of Case \ref{cs2}. In addition, by the choice of $p$, $v_{i_q-1}$ must not be in $A_s$ and set $S=A_s\cup\{v_{i_q-1},v_k\}$. Then (\ref{x1}) of (\ref{x2})  holds. So $S$ is a good set.  If $i_q=i_p-1$ and $i_q=1$, i.e.,  $p=1,q=0$ and $i_p=2$. Then by (d),  we have $A_1=A_0 \cup \{v_1\}\cup (e_{i_1} \setminus W)$, which implies that  $e_{i_1}\setminus W\subseteq A_0= e_{i_0}\setminus W$. Hence by \eqref{eq4} and $e_{i_1}\neq e_{i_0}$, we have $e_{i_1}\setminus W=e_{i_0}\setminus W$ and $e_{i_1}\cap W\neq e_{i_0}\cap W.$  Let $v_t\in (e_{i_1}\setminus e_{i_0})\cap W$ and  $(e_{i_0}\cap W)\cup \{v_t\}:=\{v_{j_0},v_{j_1},\dots,v_{j_{s+1}}\}$ for  $1=j_0<j_1<\dots<j_{s+1}$. Define recursively $B_0,B_1,\dots,B_{s+1}$. Let $B_0:=e_1\setminus W$, and for $l\in\{1,2,\dots,s+1\}$,
\begin{equation}
	B_l=\left\{
	\begin{array}{lll}
		&B_{l-1}\cup (e_{j_l}\setminus W), & B_{l-1}\cap (e_{j_l}\setminus W)=\emptyset; \\
		&B_{l-1}\cup \{v_{j_l-1}\}, &{\text{otherwise}}.
	\end{array}
	\right.
\end{equation}
Recalling that $e_2$ is the first edge of a Berge path  $P':=x e_2 v_1 e_1 v_2 e_3 v_3 \dots v_{k-1} e_k v_k$ for some $x\in e_2\setminus W$ in  $\mathcal{BP}_k$. Note that $E(P')=E(P)$. Let $W'=V(P')\setminus \{x\}$. Then $W'=W$. Since $v_t\in e_2\cap W'$ and $e_2\setminus W'\neq\emptyset$, we can apply Lemma \ref{le3} (\ref{l3c1.5}) on the Berge path $P'$, and conclude that $N(e_t\setminus W')\subseteq E(P')$, which is rewritten as $N(e_t\setminus W)\subseteq E(P)$. So $N(e_{j_l}\setminus W)\subseteq E(P)$ for each $l$. Clearly $B_{l-1}\cap W\subseteq\{v_{j_0-1},v_{j_1-1},\dots,v_{j_{l-1}-1}\}$, so $|B_l|>|B_{l-1}|$ always holds, and $|B_{s+1}|\geq |B_0|+s+1=r$. Since $B_{s+1}\cap W\subseteq \{v_{j_0-1},v_{j_1-1},\dots,v_{j_{s+1}-1}\}$, we have $v_k\notin B_{s+1}$. Let $S=B_{s+1} \cup \{v_k\}$. Then either (\ref{x1}) or (\ref{x2}) 
holds. So $S$ is a good set. We finish our proof.
\end{proof}

\begin{Lemma}
	\label{le1} Let  $\mathcal{H}$ be an $r$-uniform hypergraph.  For a given Berge path
$P$  with a terminal vertex  $v_0$, denote by $\mathcal{P}(P,v_0)$  the set of Berge paths obtained by rearranging the defining vertices and defining edges of $P$ with one terminal vertex $v_0$, i.e., $ \mathcal{P}(P,v_0):=
\{Q :  Q$  is a Berge path with one terminal vertex $ v_0$  with  $V (Q)=V(P)$  and $ E(Q)= E(P) \}$.
Then, there exists a nonempty set $\mathcal{P'}\subseteq \mathcal{P}(P,v_0)$ such that $$|N_{ E(P) }(\tau(\mathcal{P'}))|\leq 2|\tau(\mathcal{P'})|-1,$$  where  $\tau (P')$ is the other terminal  vertex of $P'$ other than $v_0$ for $P'\in \mathcal{P}(P,v_0)$ and
$\tau(\mathcal{P}^{\prime})=\{\tau(P^{\prime}): P^{\prime}\in \mathcal{P}^{\prime}\}$.
\end{Lemma}
\begin{proof}
	Let $P=v_0 e_1 v_1 e_2 v_2 \cdots v_{l-1} e_l v_l$ be a Berge path with $l=| E(P) |$.  Let $\mathcal{P'}$ be a maximal subset of $\mathcal{P}(P,v_0)$ that satisfies the following two conditions:
	\begin{enumerate}[p1)]
		\item $P\in\mathcal{P'}$;
		\item  If $v_{j-1}, v_j\notin \tau(\mathcal{P'})$  for some $1\le j\le l$, then  $v_{j-1}e_{j}v_{j}$ is a segment of $P^{\prime}$ for each Berge path $P^{\prime}\in \mathcal{P}^{\prime}$.
	\end{enumerate}
	Clearly, $\{P\}$ is a subset of $\mathcal{P}(P,v_0)$ that satisfies p1) and p2). Hence $\mathcal{P'}\neq \emptyset$.
	Further, we have the following Claim 1:

{\bf Claim 1.} If $v_{j-1}, v_j\notin \tau(\mathcal{P'})$  for $1\le j\le l$, then $\tau(\mathcal{P'})\cap e_j=\emptyset$. 

	{\bf Proof of Claim 1}.
	We prove Claim 1 by  contradiction.  Suppose that  $\tau(\mathcal{P'})\cap e_j\neq\emptyset$  and there is a vertex  $v\in \tau(\mathcal{P'})\cap e_j$. Then there is a Berge path  $Q\in \mathcal{P}^{\prime}$ with  $\tau(Q)=v$. 	By p2),  $v_{j-1}e_j v_j$ is a segment of $Q$. So we may  split $Q$ into three segments as $Q=Q_{v_0 \sim v_{j-1}}e_j Q_{v_j \sim v}$.  By  $v\in e_j$, $Q^{\prime}:=Q_{v_0 \sim v_{j-1}}e_j Q_{v \sim v_j}$ is also a Berge path with two terminal vertices $v_0$ and $v_j$.  Hence,  $Q^{\prime}\notin \mathcal{P}^{\prime}$, since $\tau(Q^{\prime})=v_j\notin \tau(\mathcal{P}^{\prime})$.
	
	Now we complete the proof of Claim 1  by proving that $\mathcal{P'}\cup\{Q^{\prime}\}$ with $\tau(\mathcal{P'}\cup\{Q^{\prime}\})=\tau(\mathcal{P'})\cup\{v_j\}$  satisfies p1) and p2), which contradicts the maximality of $\mathcal{P'}$.

	Clearly $P\in \mathcal{P'}$ and  $\mathcal{P'}\cup\{Q^{\prime}\}$ satisfies p1).
Suppose that $	v_{t-1},v_t\notin \tau(\mathcal{P'}\cup\{Q^{\prime}\})=\tau(\mathcal{P'})\cup\{v_j\}$
for some $1\le t\le l$. Then $t\neq j$ and $t\neq j+1$. By $v_{t-1}, v_t\notin \tau(\mathcal{P}^{\prime})$ and   property p2) of $\mathcal{P'}$, $v_{t-1}e_t v_t$ is a segment of each Berge path in $\mathcal{P'}$.  By  property p2) of $\mathcal{P'}$, $v_{t-1}e_t v_t$ is  a segment of $Q=Q_{v_0 \sim v_{j-1}}e_j Q_{v_j \sim v}\in  \mathcal{P'}$. By $t\neq j$,  $v_{t-1}e_t v_t$ is either a segment of $Q_{v_0 \sim v_{j-1}}$ or $Q_{v_j \sim v}$. In both cases, $v_{t-1}e_t v_t$ is a segment of $Q^{\prime}=Q_{v_0 \sim v_{j-1}}e_j Q_{v \sim v_j}$.  So $\mathcal{P}^{\prime}\cup\{Q^{\prime}\}$ satisfies p1) and p2).
\QEDB
	
	Let $I:=\{i:v_i\in \tau(\mathcal{P'})\}$  and $\mathcal{K}:=\{e_i:i\in I\cap[1,l]\}\cup\{e_i:i\in (I+1)\cap[1,l]\}$, where $I+1:=\{i+1, i\in I\}$ and $[1,l]:=\{1, \ldots, l\}$. Clearly $\mathcal{K}\subseteq N_{ E(P) }(\tau(\mathcal{P'}))$. Further, we have the following Claim 2:

{\bf Claim 2.}  $N_{ E(P) }(\tau(\mathcal{P'})) =\mathcal{K}$.
	
 {\bf Proof of Claim 2.} If there is an $e_j\in N_{ E(P) }(\tau(\mathcal{P'}))\setminus \mathcal{K}$, then  $e_j\notin \mathcal{K}$.  So by the definition of $\mathcal{K}$ and $I$,  $v_{j-1}\notin \tau(\mathcal{P'})$ and $v_j \notin \tau(\mathcal{P'})$. Hence, by Claim 1,  $e_j\notin N(\tau(\mathcal{P'}))$, which contradicts to  $e_j\in N_{ E(P) }(\tau(\mathcal{P'}))$. So Claim 2 holds.\QEDB

	Since  $v_l=\tau(P)\in \tau(\mathcal{P'})$ and  
$|(I+1)\cap [1,l]|\leq |I|-1$, we have
	\begin{equation}
		\begin{aligned}
			|N_{ E(P) }(\tau(\mathcal{P'}))|&=|\mathcal{K}|\leq |I\cap [1,l]|+|(I+1)\cap [1,l]|\\
			&\leq |I|+|I|-1=2|I|-1
			=2|\tau(\mathcal{P'})|-1. \nonumber
		\end{aligned}
	\end{equation}
 So we finish our proof.
\end{proof}

\begin{Theorem}
	\label{thy}
	Let $\mathcal{H}$  be an  $n$-vertex,  $r$-uniformly connected hypergraph with a longest Berge path of length $k$.  If $k>r\geq 3$, then there is a good set $S$ with either $|S|<n$  or $n=k+1$. Moreover, $\mathcal{S}_r(\mathcal{H})\neq\emptyset$.\end{Theorem}

\begin{proof}
Let $P=v_0 e_1 v_1 e_2 v_2 \cdots v_{k-1} e_k v_k$ be a longest Berge path of length $k$ in $\mathcal{H}$. Denote by $\mathcal{P}(P,v_0)$  the set of Berge paths obtained by rearranging the defining vertices and edges of $P$ with one terminal vertex $v_0$.
By  Lemma \ref{le1}, there is  a nonempty set $\mathcal{P'}\subseteq \mathcal{P}(P,v_0)$ such that $|N_{ E(P) }(\tau(\mathcal{P'}))|\leq 2|\tau(\mathcal{P'})|-1$. Clearly $p(e)=k$ for  each $ e\in N(\tau(\mathcal{P'}))$.
If $|N(\tau(\mathcal{P'}))|\leq f_r(k) |\tau(\mathcal{P'})|$, then $\tau(\mathcal{P'})$ is a good set  and $\mathcal{S}_r(\mathcal{H})\neq \emptyset$.
 Moreover,  $|\tau(\mathcal{P'})|<n$ by  $v_0\notin \tau(\mathcal{P'})$. Hence, we assume that
	$|N(\tau(\mathcal{P'}))|>f_r(k)|\tau(\mathcal{P'})|$.
Then by Lemma \ref{le1}, we have
\begin{equation}
	\begin{aligned}
		|N_{E(\mathcal{H})\setminus  E(P) }(\tau(\mathcal{P'}))|&=|N(\tau(\mathcal{P'}))|-|N_{ E(P) }(\tau(\mathcal{P'}))|> (f_r(k)-2)|\tau(\mathcal{P'})|+1.
	\end{aligned}
\end{equation}
Hence, there is at least one vertex $v_a$ in $\tau(\mathcal{P'})$  incident to more than $f_r(k)-2$ edges in $E(\mathcal{H})\setminus  E(P) $.
So there is a Berge path $P^{\prime}\in {\mathcal{P}}(P, v_0)$ with the terminal vertex $v_a$.  Let $\mathcal{P}(P',v_a)$  be the set of Berge paths obtained by rearranging the defining vertices and edges of $P^{\prime}$ with one terminal vertex $v_a$. By
Lemma \ref{le1}, there is a nonempty set $\mathcal{P''}\subseteq \mathcal{P}(P',v_a)$ such that $|N_{ E(P) }(\tau(\mathcal{P''}))|\leq 2|\tau(\mathcal{P''})|-1$. If $|N(\tau(\mathcal{P''}))|\leq f_r(k) |\tau(\mathcal{P''})|$, then $\tau(\mathcal{P''})$ is a good set.
Moreover, by $v_0\notin \tau(\mathcal{P''})$,  we have $|\tau(\mathcal{P''})|<n$. So the assertion holds.
Hence, we assume that 	$|N(\tau(\mathcal{P''}))|>f_r(k) |\tau(\mathcal{P''})|$.
Then
\begin{equation}
	\begin{aligned}
		|N_{E(\mathcal{H})\setminus  E(P) }(\tau(\mathcal{P''}))|&=|N(\tau(\mathcal{P''}))|-|N_{ E(P) }(\tau(\mathcal{P''}))|> (f_r(k)-2)|\tau(\mathcal{P''})|+1.
	\end{aligned}
\end{equation}
Hence, there is at least one vertex $v_b$ in $\tau(\mathcal{P^{\prime\prime}})$  incident to more than $f_r(k)-2$ edges in $E(\mathcal{H})\setminus  E(P) $.
So there is a Berge path $P^{\prime\prime}\in {\mathcal{P}}(P^{\prime}, v_a)$ with the other terminal vertex $v_b$.
 Hence, there is  a Berge Path $P^{\prime\prime}$ with two terminal vertices $v_a$ and $v_b$  which are incident to
   more than $f_r(k)-2$ edges in $E(\mathcal{H})\setminus  E(P) $.
Without loss of generality, we may assume that two terminal vertices $v_0$ and $v_k$ of $P$  are incident to more than $f_r(k)-2$ edges in $E(\mathcal{H})\setminus  E(P) $.

If there exists an	edge $ h\in E(\mathcal{H})\setminus  E(P) $ with $\{v_0,v_k\}\subseteq h$, then there is a Berge cycle of $k+1$ edges: $C:=v_0 e_1 v_1 \dots v_{k-1} e_k v_k h v_0$. Thus, by Lemma \ref{le4},  we have $n=|V(\mathcal{H})|=|V(C)|=k+1$.  Then by $k>r$ and $|N(V(\mathcal{H}))| \le |E(\mathcal{H})|\le \binom{k+1}{r}=f_r(k)|V(\mathcal{H})|$,   $V(\mathcal{H})$ is a good set with $n=k+1$. 
Hence, we may assume that either $v_0\notin h$ or $v_k\notin h$ for each edge $ h\in E(\mathcal{H})\setminus  E(P) $.
Note that for each $e\in N_{E(\mathcal{H})\setminus E(P)}(\tau(\mathcal{P}))$, there must be $e\subseteq V(P)$. Otherwise, we can extend $P$ to a longer Berge path by adding $e$ and a vertex in $e\setminus V(P)$,  which contradicts the maximality of $P$.
Let
$$I=\{i: v_i\in \cup\{e: \ e\in N_{E(\mathcal{H})\setminus  E(P) }(v_0)\}\}\setminus \{0\}$$
  and $$J=\{j: v_j\in \cup\{e:\ e\in N_{E(\mathcal{H})\setminus  E(P) }(v_k)\}\}\setminus\{k\}.$$
    Then  $I\subseteq [1,k-1]$, $J\subseteq [1,k-1]$, and
\begin{equation}
	\label{i2}
	|I\cup (J+1)|\leq |[1,k]|=k,\nonumber
\end{equation}
where $J+1:=\{j+1: j\in J\}$ and $[1,k-1]=\{1, \ldots, k-1\}$.
Further we have the following Claim.

 {\bf Claim:}  Either  $I\cap (J+1)\neq \emptyset$ or  there is  a Berge cycle of length $k+1$ in $\mathcal{H}$.

{\bf Proof of Claim.}  We consider the following three cases.

\setcounter{case}{0}
\begin{case}
	$k\geq r+1\geq 5$, or $k\geq r+3=6$.
\end{case}

Then it is directly to prove  (see Appendix) that
\begin{equation}\label{inequality}
	f_r(k)-2\geq \binom{\frac{k}{2}}{r-1}.
\end{equation}
So we have
\begin{equation}
	\label{eqn}
	|N_{E(\mathcal{H})\setminus E(P) }(v_0)|>\binom{\frac{k}{2}}{r-1},\ \ |N_{E(\mathcal{H})\setminus E(P) }(v_k)|>\binom{\frac{k}{2}}{r-1}.
\end{equation}
Hence, $\cup\{e: e\in N_{E(\mathcal{H})\setminus  E(P) }(v_0)\}$  consists of more than $\frac{k}{2}$ vertices other than $v_0$, and
  $ \cup \{e: \ e\in N_{E(\mathcal{H})\setminus  E(P) }(v_k)\}$ consists of more than $\frac{k}{2}$ vertices other than $v_k$. By the maximality of $P$, all these vertices must be in $V(P)$.
So
\begin{equation}
	\label{i1}
	|I|>\frac{k}{2},\ \ |J|> \frac{k}{2}.
\end{equation}
Hence by \eqref{i1} and \eqref{i2}, we have 
\begin{equation}
		|I\cap (J+1)|=|I|+|J+1|-|I\cup (J+1)|
		\geq |I|+|J|-k
		> \frac{k}{2}+\frac{k}{2}-k= 0.\nonumber
\end{equation}
So $I\cap (J+1)\neq\emptyset$.

\begin{case}
	$k= r+2=5$.
\end{case}
It is  easy to see  that
\begin{equation}
	f_r(k)-2= \frac{10}{3}-2=\frac{4}{3}.
\nonumber
\end{equation}
So we have
\begin{equation}
	|N_{E(\mathcal{H})\setminus E(P) }(v_0)|,|N_{E(\mathcal{H})\setminus E(P) }(v_k)|\ge 2
\end{equation}
Hence  $\cup\{e: e\in N_{E(\mathcal{H})\setminus  E(P) }(v_0)\}$  consists of more than  two 
vertices other than $v_0$,  and   $\cup\{e: e\in N_{E(\mathcal{H})\setminus  E(P) }(v_k)\}$  consists of  more than two vertices other than $v_k$. Furthermore, by the maximality of $P$, all these vertices must be in $V(P)$.
Hence,
\begin{equation}
	\label{i11}
	|I|\ge 3,\ \ |J|\geq 3.
\end{equation}
By \eqref{i11} and \eqref{i2}, 
\begin{equation}
	\begin{aligned}
		|I\cap (J+1)|&=|I|+|J+1|-|I\cup (J+1)|
		\geq |I|+|J|-k
		\geq 3+3-5= 1. \nonumber
	\end{aligned}
\end{equation}
So $I\cap (J+1)\neq\emptyset$.

\begin{case}
	$k= r+1=4$.
\end{case}
Then 
	$f_r(k)-2=0,$ 
 which implies that
\begin{equation}
	|N_{E(\mathcal{H})\setminus E(P) }(v_0)|>0\  \text{and} \ \ |N_{E(\mathcal{H})\setminus E(P) }(v_4)|>0.
\end{equation}
Hence  $\cup\{e: e\in N_{E(\mathcal{H})\setminus  E(P) }(v_0)\}$   consists of at least two vertices other than $v_0$,  $\cup\{e: e\in N_{E(\mathcal{H})\setminus  E(P) }(v_k)\}$  consists of at least two vertices other than $v_4$. Then by  the maximality of $P$, all these vertices must be in $V(P)$.
So
\begin{equation}
	\label{i111}
	|I|\ge 2,\ \ \ |J|\geq 2.
\end{equation}
By \eqref{i111} and  \eqref{i2},
\begin{equation}
		|I\cap (J+1)|=|I|+|J+1|-|I\cup (J+1)|
		\geq |I|+|J|-k\geq 2+2-4= 0.
\end{equation}

By $k=4$ and  the definition of $I$ and $J$,  $I\subseteq \{1,2,3\}$ and $J+1\subseteq \{2,3,4\}$.  Hence if $1\notin I$ or $3\notin J$,  then $I\cup (J+1) \subseteq \{2,3,4\}$ or $I\cup (J+1) \subseteq \{1,2,3\}$. So $|I\cap (J+1)|=|I|+|J+1|-|I\cup (J+1)|\geq |I|+|J|-3\geq  1$.  If $1\in I$  and $3\in J$, then there exists an edge $e'_1\in N_{E(\mathcal{H})\setminus E(P) }(v_0)$ with  $\{v_0,v_1\}\subseteq e'_1$. In addition, note that $\{v_0,v_1\}\subseteq e_1$, and $e_1,e'_1\subseteq \{v_0,v_1,v_2,v_3\}$.  So $\{e_1,e'_1\}=\{\{v_0,v_1,v_2\},\{v_0,v_1,v_3\}\}$.
Similarly, there exists  an edge $e'_4\in N_{E(\mathcal{H})\setminus E(P) }(v_4)$ containing both $v_3$ and $v_4$. So $\{e_4,e'_4\}=\{\{v_1,v_3,v_4\},\{v_2,v_3,v_4\}\}$.
Hence, $\mathcal{H}$ contains a Berge cycle of $k+1=5$ edges, $v_0 \{v_0,v_1,v_2\} v_2 e_3 v_3 \{v_2,v_3,v_4\} v_4 \{v_1,v_3,v_4\} v_1 \{v_0 v_1 v_3\} v_0$.
Therefore, the claim holds.
\QEDB

By Claim, If $\mathcal{H}$ has a Berge cycle of length $k+1$, then by Lemma \ref{le4},  $V(\mathcal{H}) $ is a good set  and $n=k+1$.    Now we assume that   $I\cap (J+1)\neq \emptyset$. Then there is a $t\in I\cap (J+1)$. Hence, there exists  two  edges $e, h\in E(\mathcal{H})\setminus  E(P) $ with $\{v_0,v_t\}\subseteq e$ and $\{v_k, v_{t-1}\}\subseteq h$. Hence, $\mathcal{H}$ contains a Berge cycle of $k+1$ edges, $v_0 e_1 v_1 \dots v_{t-1} h  v_k e_k v_{k-1}\dots v_t e v_0$. By Lemma \ref{le4},  $V(\mathcal{H}) $ is a good set  and $n=k+1$.
\end{proof}

\section{Proof of Theorem~\ref{th3conn}}

In order to prove Theorem~\ref{th3conn}, we also need the following result

\begin{Lemma}[Gy\H{o}ri, Lemons, Salia, Zamora \cite{gyHori2021structure}]
	\label{le6}
	Let $\mathcal{H}$ be an $r$-uniform hypergraph with $r\geq 4$ and $|V(\mathcal{H})|=r+1>|E(\mathcal{H})|\geq 2$. Then every pair of vertices of $\mathcal{H}$ are joined by a Berge path of length $|E(\mathcal{H})|$.
\end{Lemma}

\begin{Corollary}
	\label{coro-path}
	Let $\mathcal{H}$ be an $r$-uniform hypergraph with $|V(\mathcal{H})|=r+1$ and $1\leq |E(\mathcal{H})|< r$. Then for every vertex $v\in V(\mathcal{H})$, $\mathcal{H}$ contains a Berge path of length $|E(\mathcal{H})|$ starting at $v$.
\end{Corollary}
\begin{proof}
If $r\geq 4$ and $|E(\mathcal{H})|\geq 2$,  then by Lemma \ref{le6}, the assertion holds. If  $|E(\mathcal{H})|=1$ or $2\leq r\leq 3$, the assertion clearly holds.
\end{proof}

\noindent\emph{Proof of} \textbf{Theorem \ref{th3conn}}.


We prove \eqref{eq0} by the induction on $n$. If $n=r$, then $|V(\mathcal{H})|=r$ and $E(\mathcal{H})=\{e\}$. So  $p_{\mathcal{H}}(e)=1$, $f_r(p_{\mathcal{H}}(e))=\frac{1}{r}$ and
$\sum_{e\in E(\mathcal{H})}\frac{1}{f_r(p_{\mathcal{H}}(e))}=r$. Hence \eqref{eq0} holds.
      Suppose that  \eqref{eq0}  holds for all $n'<n$.
            By Theorems \ref{thx} and \ref{thy},  we have $\mathcal{S}_r(\mathcal{H})\neq\emptyset$.  There is an  $S\in \mathcal{S}_r(\mathcal{H})$. Then
\begin{equation}
	\begin{aligned}
		\label{eq1}
		\sum_{e\in E(\mathcal{H})}\frac{1}{f_r(p_{\mathcal{H}}(e))}=&\sum_{e\in E(\mathcal{H}\setminus S)}\frac{1}{f_r(p_{\mathcal{H}}(e))} + \sum_{e\in N(S)}\frac{1}{f_r(p_{\mathcal{H}}(e))}.\\
	\end{aligned}
\end{equation}
Note that
 $p_{\mathcal{H}}(e)\geq p_{\mathcal{H}\setminus S}(e)$ for each $e\in  E(\mathcal{H} \setminus S)$.  By the monotonicity of the function $f_r(x)$ and the induction hypothesis,
\begin{equation}
	\label{eq2}
	\sum_{e\in E(\mathcal{H}\setminus S)}\frac{1}{f_r(p_{\mathcal{H}}(e))}\leq \sum_{e\in E(\mathcal{H}\setminus S)}\frac{1}{f_r(p_{\mathcal{H}\setminus S}(e))}\leq n-|S|.
\end{equation}
On the other hand, by the definition of $S\in \mathcal{S}_r(\mathcal{H})$,   $p_{\mathcal{H}}(e)=k$ for each $e\in N(S)$,  where $k>1$ is the maximum length of Berge path in $\mathcal{H}$. Hence by  the definition of $\mathcal{S}_r(\mathcal{H})$,  we have
\begin{equation}
	\label{eq3}
	\sum_{e\in N(S)}\frac{1}{f_r(p_{\mathcal{H}}(e))}=\sum_{e\in N(S)}\frac{1}{f_r(k)}=\frac{|N(S)|}{f_r(k)}\leq |S|.
\end{equation}
Hence by follows from  \eqref{eq1}, \eqref{eq2} and \eqref{eq3},  we have \eqref{eq0} for $n$.

Furthermore, if $n=r+1$ and $|E(\mathcal{H})|\in \{1, 2, \ldots, r-1, r+1\}$ for $r\ge 3$, then for each $e\in E(\mathcal{H})$,  $$p_{\mathcal{H}}(e)=\left\{\begin{array}{ll}
|E(\mathcal{H})|, & \text{for}\ |E(\mathcal{H})|\le r-1;\\
r, &\text{for}\ |E(\mathcal{H})|= r+1.
   \end{array}\right.$$
 Hence, for each $e\in E(\mathcal{H})$,
    $$   f_r(p_{\mathcal{H}}(e))=\left\{\begin{array}{ll}
    \frac{|E(\mathcal{H})|}{r+1}, &\text{ for} \ |E(\mathcal{H})|\le r-1;\\
    \frac{1}{r}\binom{r}{r-1}, &\text{for}\ |E(\mathcal{H})|\le r+1.
    \end{array}\right.$$
 Therefore,
 $$
 \sum_{e\in E(\mathcal{H})}\frac{1}{f_r(p_{\mathcal{H}}(e))}=\left\{\begin{array}{ll}
 r+1, & \text{for}\  |E(\mathcal{H})|\le r-1;\\
  (r+1)\frac{1}{r}\binom{r}{r-1}, & \text{for}\  |E(\mathcal{H})|= r+1.\end{array}\right.$$
So	 \eqref{eq0} becomes equality.

If  $n\neq r+1
$ and $|E(\mathcal{H})|=\binom{n}{r}$, then for each $e\in E(\mathcal{H})$,
$$p_{\mathcal{H}}(e)=\left\{\begin{array}{ll}
1, & \text{for}\ n=r;\\
n-1, &\text{for}\ n\ge r+2.
   \end{array}\right.$$
   Hence, for each $e\in E(\mathcal{H})$,
   $$   f_r(p_{\mathcal{H}}(e))=\left\{\begin{array}{ll}
    \frac{1}{r}, &\text{ for} \ n=r;\\
    \frac{1}{r}\binom{n}{r-1}, &\text{for}\ n\ge r+2.
    \end{array}\right.$$
     Therefore,
 $$
 \sum_{e\in E(\mathcal{H})}\frac{1}{f_r(p_{\mathcal{H}}(e))}=\left\{\begin{array}{ll}
 r, & \text{for}\  n=r;\\
  n, & \text{for}\  n\ge r+2.\end{array}\right.$$
  So	 \eqref{eq0} becomes equality.

   Conversely,   for $n\ge r\ge 3,$   we use the induction on $n$ to  prove the following assertion that if
	\begin{equation}
		\label{eqlt}
		\sum_{e\in E(\mathcal{H})}\frac{1}{f_r(p_{\mathcal{H}}(e))}=n.
	\end{equation}
	then  one of the following holds:
	\begin{enumerate}[(i)]
		\item $n=r+1$ and $|E(\mathcal{H})|\in \{1, 2, \ldots, r-1, r+1\}$.
		\item $n\neq r+1$ and $|E(\mathcal{H})|=\binom{n}{r}$.
	\end{enumerate}

Clearly, if $n=r$, then $|E(\mathcal{H})|=1$. So the assertion holds. If $n=r+1$, it is easy to see that $|E(\mathcal{H})|\neq r$.  Otherwise,  $|E(\mathcal{H})|= r$, and $p_{\mathcal{H}}(e)=r$ for each $e\in E( \mathcal{H})$. So By \eqref{eqlt},
$$n=\sum_{e\in E(\mathcal{H})}\frac{1}{f_r(p_{\mathcal{H}}(e))}=\sum_{e\in E(\mathcal{H})}\frac{1}{f_r(r)}=r\neq n,$$
which is a contradiction. Hence, (i) holds.

Suppose that the assertion holds for less than $n$. We consider the case of $n$. Then we have the following Claim:

      {\bf Claim: } $\mathcal{S}_r(\mathcal{H})= \{V(\mathcal {H})\}$.

{\bf Proof of Claim.} Suppose that there is an  $|S|<|V(\mathcal{H})|=n$  with an $S\in \mathcal{S}_r(\mathcal{H})$.
By \eqref{eqlt} and \eqref{eq2}, we have
\begin{equation}
	\label{eq14}
	\sum_{e\in E(\mathcal{H}\setminus S)}\frac{1}{f_r(p_{\mathcal{H}\setminus S}(e))}= n-|S|=|V(\mathcal{H}\setminus S)|,
\end{equation}
and
\begin{equation}
	\label{eq26}
	p_{\mathcal{H}}(e)=p_{\mathcal{H}\setminus S}(e), \forall e\in  E(\mathcal{H}\setminus S).
\end{equation}
Let $\mathcal{C}_1,\mathcal{C}_2,\dots,\mathcal{C}_t$ be the components of $\mathcal{H}\setminus S$. Then
\begin{equation}
	\label{eq28}
	\sum_{e\in E(\mathcal{H}\setminus S)}\frac{1}{f_r(p_{\mathcal{H}}(e))}=\sum_{i=1}^{t}\sum_{e\in E(\mathcal{C}_i)}\frac{1}{f_r(p_{\mathcal{H}}(e))}\leq \sum_{i=1}^{t}|V(\mathcal{C}_i)|=|V(\mathcal{H}\setminus S)|.
\end{equation}
Hence 
by \eqref{eq14} and \eqref{eq28}, 
\begin{equation}
	\sum_{e\in E(\mathcal{C}_i)}\frac{1}{f_r(p_{\mathcal{H}\setminus S}(e))}= |V(\mathcal{C}_i)|, \ \ i=1, \ldots, t.
\end{equation}
By the connectivity of $\mathcal{H}$, there is an edge $h_i$  with  $h_i\cap V(\mathcal{C}_i)\neq \emptyset$ and $h_i\cap S\neq\emptyset$. There exist two vertices $u_i$ and $v_i$ with  $u_i\in h_i\cap V(\mathcal{C}_i)$ and  $v_i\in h\cap S$. By \eqref{eq28}, $|E(\mathcal{C}_i)|\ge 1$, which implies that $r\le |V(\mathcal{C}_i)|\leq |V(\mathcal{H}\setminus S)|=n-|S|<n$. Hence, by   the induction hypothesis, $\mathcal{C}_i$ satisfies either (i) or (ii).  If (i) holds, then either  $|E(\mathcal{C}_i)|<r$ or $|E(\mathcal{C}_i)|=r+1$.  When $|E(\mathcal{C}_i)|<r$, by Corollary \ref{coro-path}, there is  a Berge path $P_i$ starting from $u_i$, whose defining edges are $E(\mathcal{H}\setminus S)$. So $v_i h_i P_i$ is a longer Berge path  than $P_i$. So $p_{\mathcal{H}}(e)>p_{\mathcal{H}\setminus S}(e)$ for each $e\in E(P_i)\subseteq E(\mathcal{H}\setminus S)$,  which contradicts to  \eqref{eq26}.  When $|E(\mathcal{C}_i)|=r+1$, it is easy to see that there is an edge $e\in E(P)\subseteq E(\mathcal{H}\setminus S)$ such that $p_{\mathcal{H}}(e)>p_{\mathcal{H}\setminus S}(e)$, which contradicts to \eqref{eq26}.
If (ii) holds, then
	$|E(\mathcal{C}_i)|=\binom{|V(\mathcal{C}_i)|}{r}$.
Hence,  $\mathcal{C}_i$ contains a Berge path $P_i$ which has $|V(\mathcal{C}_i)|$ defining vertices, and starts from $u_i$. So $v_i h_i P_i$ is a Berge path with more defining vertices than $P_i$. Then $p_{\mathcal{H}}(e)>p_{\mathcal{H}\setminus S}(e)$ for each $e\in E(P_i)\subseteq E(\mathcal{H}\setminus S)$, which contradicts to \eqref{eq26}. Therefore, the claim holds. \QEDB

Let $k$ be the length of the longest Berge path in $\mathcal{H}$.  By the definition of $\mathcal{S}_r(\mathcal{H})= \{V(\mathcal {H})\}$,  
$p_{\mathcal{H}}(e)=k$ for each $e\in E(\mathcal{H})$. Then By \eqref{eqlt}, $|E(\mathcal{H})|= f_r(k)|V(\mathcal{H})|=f_r(k)n\ge 1.$

If $n=r$, then $|E(\mathcal{H})|=1=\binom{n}{r}$ and (ii) holds. If $n=r+1$, then $|E(\mathcal{H})|\neq r$. Otherwise, $p_{\mathcal{H}}(e)=r$ and
$$\sum_{e\in E(\mathcal{H})}\frac{1}{f_r(p_{\mathcal{H}}(e))}=\sum_{e\in E(\mathcal{H})}\frac{1}{f_r(r)}=r\neq n,$$
which is a contradiction. So (i) holds.
Hence, we assume that $n\ge r+2$. Furthermore, we can prove that $k>r$.

If fact,  if $k=r$, then by  Claim and Theorem \ref{thx}, $n=|V(\mathcal{H})|\ge r+1$ and $|E(\mathcal{H})|=|N(V(\mathcal{H}))|\le r+1$. On the other hand, $|E(\mathcal{H})|=f_r(k)n= f_r(r)n=n\ge r+2.$ It is a contradiction. If $1<k<r$, then  by  Claim and Theorem \ref{thx}, $n=|V(\mathcal{H})|\ge r+1$ and $|E(\mathcal{H})|=|N(V(\mathcal{H}))|\le k$. On the other hand, $|E(\mathcal{H})|=f_r(k)n= f_r(k)n=n\ge r+2.$ It is a contradiction.
If $k=1$, then $\mathcal{H}$ is disconnected, which is a contradiction.

  So $k>r$. Then by Theorem \ref{thy} and Claim, we have  $n=k+1$, and  $|E(\mathcal{H})|=\frac{1}{r}\binom{k}{r}=\binom{n}{r}$. Hence $\mathcal{H}$ is a complete hypergraph on $n=k+1\geq r+2$ vertices. So (ii) holds.
\QEDB

\appendix

\section{Appendix }

In the appendix, we present a proof of \eqref{inequality}, i.e.,
if $k\geq r+1\geq 5$, or $k\geq r+3=6$, then
\begin{equation}\nonumber
	f_r(k)-2\geq \binom{\frac{k}{2}}{r-1}.
\end{equation}

{\bf Proof.} For an positive integer $m$ and a real number $x\geq m$, denote
\begin{equation}
	(x)_m:=x(x-1)\dots (x-m+1).\nonumber
\end{equation}
Let
\begin{equation}
	g(k)=f_r(k)-\binom{\frac{k}{2}}{r-1}=\frac{1}{r}\binom{k}{r-1}-\binom{\frac{k}{2}}{r-1}.\nonumber
\end{equation}
If $r=3$ and $k\ge 6$, then
$$f_r(k)-2-\binom{\frac{k}{2}}{r-1}=\frac{1}{3}\binom{k}{2}-2-\binom{\frac{k}{2}}{2}=\frac{(k+8)(k-6)}{24}\ge 0.$$
So the assertion holds.
 Hence it suffices to show that $g(k)\geq 2$ for $k\ge r+1\ge 5$. We consider the following two cases.

\setcounter{case}{0}
\begin{case}
	$\frac{k}{2}<r-1$.
\end{case}
Then by $k\geq r+1\ge 5$,
\begin{equation}
	g(k)=\frac{1}{r}\binom{k}{r-1}\geq g(r+1)=\frac{1}{r}\binom{r+1}{r-1}=\frac{r+1}{2}\geq 2. \nonumber
\end{equation}

\begin{case}
	$\frac{k}{2}\geq r-1$, i.e. $k\geq 2r-2$.
\end{case}

Then,  by $k\geq r+1\ge 5$,
\begin{equation}
	\begin{aligned}
g(k)-g(k-1)&=\left(\frac{1}{r}\binom{k}{r-1}-\binom{\frac{k}{2}}{r-1}\right)-\left(\frac{1}{r}\binom{k-1}{r-1}-\binom{\frac{k-1}{2}}{r-1}\right)\\
		&=\frac{1}{r!}[(k)_{r-1}-(k-1)_{r-1}]-\frac{1}{(r-1)!}[(\frac{k}{2})_{r-1}-(\frac{k-1}{2})_{r-1}]\\
		&\ge\frac{1}{r!}\frac{r-1}{k-r+1}(k-1)_{r-1}-\frac{1}{(r-1)!}\frac{r-\frac{3}{2}}{\frac{k}{2}-r+1}(\frac{k}{2}-1)_{r-1}\\
&>0,	
\end{aligned} \nonumber
\end{equation}
where the last inequality follows from that
\begin{equation}
	\begin{aligned}
\frac{\frac{1}{r!}\frac{r-1}{k-r+1}(k-1)_{r-1}}{\frac{1}{(r-1)!}\frac{r-\frac{3}{2}}{\frac{k}{2}-r+1}(\frac{k}{2}-1)_{r-1}}&=\frac{1}{r}\cdot
\frac{r-1}{r-\frac{3}{2}}\cdot \frac{(k-1)(k-2)\dots (k-r+2)}{(\frac{k}{2}-1)(\frac{k}{2}-2)\dots (\frac{k}{2}-r+2)} 
		\geq \frac{1}{r}\cdot\frac{r-1}{r-\frac{3}{2}}\cdot 2^{r-2}
		>1.
	\end{aligned} \nonumber
\end{equation}
Hence,  $g(k)\ge g(k-1)\geq \cdots\ge  g(2r-2)$ for $k\geq r+1\geq 5$.
Moreover,
\begin{equation}
		g(2r-2)\geq \frac{1}{r}\binom{2r-2}{r-1}-1=\frac{2r-2}{r-1}\cdot\frac{2r-3}{r-2}\cdot\dots\cdot\frac{r+1}{2}-1
		\geq 2^{r-2}-1
		\geq 2.
\nonumber
\end{equation}
So the assertion holds.

\begin{thebibliography}{100}

\bibitem{balogh2023} J.~Balogh, D.~Brada\u{c}, B.~Lidick\`{y}, Weighted Tur\'{a}n theorems with applications to Ramsey-Tur\'{a}n type
of problems, 2023, arXiv preprint arXiv:2302.07859.
\bibitem{balogh2022} J.~Balogh, C.~Chen, G.~McCourt, C.~Murley, Ramsey-Tur\'{a}n problems with small independence numbers,
2022, arXiv preprint arXiv:2207.10545.

\bibitem{bradac2022}D.~Brada\u{c}, A generalization of Tur\'{a}n's theorem, 2022, arXiv preprint arXiv:2205.08923.

\bibitem{erdHos1959maximal}  P.~Erd{\H{o}}s, T.~Gallai, On maximal paths and circuits of graphs, Acta Math. Hungar. 10(3-4) (1959) 337--356.

\bibitem{berge1973graphs} C.~Berge, Graphs and hypergraphs, 1973.

\bibitem{davoodi2018erdHos} A.~Davoodi, E.~Gy\H{o}ri, A.~Methuku, C.~Tompkins, An Erd{\H{o}}s--Gallai type theorem for uniform hypergraphs, European J. Combin. 69 (2018) 159--162.

\bibitem{gyHori2021structure} E.~Gy\H{o}ri, N.~Lemons, N.~Salia, O. Zamora, The Structure of Hypergraphs without long Berge cycles, J. Combin. Theory Ser. B 148 (2021) 239--250.

\bibitem{gyHori2016hypergraph} E.~Gy\H{o}ri, G.~Y.~Katona, N.~Lemons, Hypergraph extensions of the Erd{\H{o}}s-Gallai theorem, European J. Combin. 58 (2016) 238--246.



\bibitem{malec2023localized} D.~Malec, C.~Tompkins, Localized versions of extremal problems, European J. Combin. 112 (2023) 103715.

\end{thebibliography}
\end{document}